\documentclass[10pt,twocolumn,twoside]{IEEEtran}

\IEEEoverridecommandlockouts                              % This command is only
\usepackage{subcaption} % ... cannot be used with subfigure
\usepackage{amssymb, geometry}
\usepackage{amsthm}
\usepackage{amsfonts, amsmath, latexsym}
\usepackage{algorithm}
\usepackage{algorithmic}

\usepackage{cite}               % ... more powerful citations
\usepackage{url}                % ... for email addresses, hypertext links
\usepackage{version}            % ... for comments
\usepackage{multirow}

\usepackage{tikz,pgfplots,pgfplotstable}
\graphicspath{{figures/}}    

\newcommand{\pagebudget}[1]{} % allocation of responsibility

\newtheorem{definition}{Definition}

\newtheorem{lemma}{Lemma}
\newtheorem{proposition}{Proposition}
\newtheorem{remark}{Remark}

\newcommand{\R}{\mathbb{R}}

\newcommand{\mini}{\mathop{\mbox{minimize}}}
\newcommand{\maxi}{\mathop{\mbox{maximize}}}
\newcommand{\amin}{\mathop{\mbox{argmin}}}

\newcommand{\st}{\mbox{subject to}}

\newcommand{\dps}{\displaystyle}

\newcommand{\bvec}{\left(\begin{array}{c}}
\newcommand{\evec}{\end{array}\right)}
\newcommand{\bsub}{\begin{subequations}}
\newcommand{\esub}{\end{subequations}}
\newcommand{\diag}{\mbox{diag}}

\newcommand{\tc}{\textcolor}

% ... 

\newcommand{\bfo}{{\bf 1}}

\newcommand{\DefinedAs}[0]{\mathrel{\mathop:}=}

\newcommand{\exclude}[1]{} % ... simply excludes stuff

\begin{document}

\author{Fu Lin \thanks{F.\ Lin is with the Systems Department, United Technologies Research Center, 411 Silver Ln, East Hartford, CT 06108. E-mail: linf@utrc.utc.com.}}

\title{\LARGE \bf Worst-Case Load Shedding in Electric Power Networks}

\date{\today}

\maketitle

\begin{abstract}
We consider the worst-case load-shedding problem in electric power networks where a number of transmission lines are to be taken out of service. The objective is to identify a pre-specified number of line outage that leads to the maximum interruption of  power generation and load at the transmission level, subject to the AC power flow model, the load and generation capacity of the buses, and the phase angle limit across the transmission lines. For this nonlinear model with binary constraints, we show that all decision variables are separable except for the nonlinear power flow equations. We develop an iterative decomposition algorithm, which converts the worst-case load shedding problem into a sequence of small subproblems. We show that the subproblems are either convex problems that can be solved efficiently or nonconvex problems that have closed-form solutions. Consequently, our approach is scalable for large networks. Furthermore, we prove global convergence of our algorithm to a critical point and the objective value is guaranteed to decrease throughout the iterations. Numerical experiments with IEEE test cases demonstrate the effectiveness of the developed approach. 
\end{abstract}

\vspace{1ex}


\vspace{1ex}

%%%%%%%%%%%%%%%%%%%%%%%%%%%%%%%%%%%%%%%%%%%%%%%%%%%%%%%%%%%%%%%%%%%%%%%%
\section{Introduction}

Redundancy of interconnection in power systems is known to help prevent cascade blackouts~\cite{and05}. On the other hand, recent study suggests that having too much interconnectivity in power networks can result in excessive capacity, which in turn fuels larger blackouts~\cite{brudsolei12}. Therefore, a balance between the operational robustness and the network interconnectivity is important for power grid operations. 

Traditionally, contingency analysis in power grids has focused on the severity of line outages using linearized power flow models~\cite{gal84}.  Recent years have seen vulnerability analysis of line outages using nonlinear power flow models~\cite{donloplespinyanmez05,donloplespinyanmez08,pinmezdonles10}. \tc{black}{Following this line of research, we study the worst-case load-shedding problem. Our objective is to identify a small number of transmission lines whose removal leads to the maximum damage to the power systems.  This problem contains binary decision variables for taking lines offline and the nonlinear AC power flow equations. As a result, it falls into the class of mixed-integer nonlinear programs (MINLPs), which is beyond the capacities of the state-of-the-art MINLP solvers even for small power systems.} 

\tc{black}{Our contributions can be summarized as follows. First, the worst-case load-shedding model incorporates the AC power flow equations, the generation and load capacities of the buses, and the thermal constraints across the transmission lines. This model is capable of providing more accurate operating conditions than models based on the linearized power flow equations. Second, we show that the decision variables for taking transmission lines offline and for the generation, load, and phase angles across buses are separable except for the power flow constraints. By exploiting this separable structure, we develop an algorithm that decomposes the worst-case load-shedding problem into a sequence of subproblems that are either convex problems or nonconvex problems that have closed-form solutions. As a result, our approach is scalable for large networks. Third, we prove the global convergence of our algorithm to a critical point of the nonconvex problem. Furthermore, the objective value is monotonically decreasing throughout the iterations. Our proof techniques build upon convergence results from the proximal alternating linearization method (PALM).}

In our previous work~\cite{linche16}, the alternating direction method of multipliers~(ADMM) was proposed to deal with the optimal load-shedding problem with linearized power flow model. The shortcoming of ADMM is that there is no theoretical guarantee of convergence for nonconvex problems. In contrast, PALM allows us to handle nonconvex, nonsmooth problem with provably convergence guarantee. 

%PALM has found applications in learning low-complexity model with steady-state data~\cite{linche17,linche18}, in sparse control of networked systems~\cite{linbop17}, and in co-design of output feedback controllers~\cite{linade17}.

There is a large body of work on the load shedding problem in electric power networks~\cite{moselhmba97,xugir01,aponel06,farpietir07,sadamrran09,hajpestinpie68,palshamis85,halkot93,ter06,pahscodassch13}. We next provide a brief literature review and put our contributions in context.

Several studies focus on the load-shedding problem in static networks, that is, the network structure does not change over time~\cite{moselhmba97,aponel06,farpietir07}. 
 In contrast, our load-shedding model allows the operator to remove a prescribed number of lines and evaluate the maximum amount of load loss.

Another line of work studies efficient numerical methods for the load-shedding problem~\cite{hajpestinpie68,palshamis85,halkot93,xugir01,ter06}. In~\cite{xugir01}, a discretization technique was developed to convert the differential equations to algebraic constraints. The resulting nonlinear programming (NLP) problem was solved by using standard NLP solvers. Alternatively, conventional optimization methods have been proposed for similar NLP formulations.  In~\cite{hajpestinpie68}, the Newton's method was employed to minimize the curtailment of load service after severe faults. In~\cite{palshamis85}, a quasi-Newton method was proposed for the load shedding problem with voltage and frequency characteristics of load. In~\cite{halkot93}, a projected gradient method was used to solve the under-frequency load shedding problem.  In contrast to these NLP-based approaches, our formulation incorporates binary decision variables to model line removals in AC power networks. Thus, it falls in the class of more general class of MINLP problems. 

Heuristics approaches have been proposed for the nonconvex load-shedding problem~\cite{sadamrran09,pahscodassch13}. In~\cite{sadamrran09}, a particle swarm-based simulated annealing technique was introduced for the under voltage load-shedding problem. In~\cite{pahscodassch13}, tree-like heuristics strategies were proposed for emergency situations to maintain reliability. In contrast to these heuristics approaches with little theoretical guarantees, we prove that PALM converges to a critical point of the nonconvex load-shedding problem. 

Our presentation is organized as follows. In Section~\ref{sec.loadshed}, we formulate the worst-case load-shedding problem for the AC power networks. In Section~\ref{sec.separable}, we study the separable structure of the load-shedding problem. In Section~\ref{sec.PALM}, we develop the PALM algorithm and in Section~\ref{sec.conv}, we analyze its convergence behavior. In Section~\ref{sec.results}, we provide numerical results for the IEEE test cases. In Section~\ref{sec.conclusion}, we conclude the paper and discuss future directions.

%%%%%%%%%%%%%%%%%%%%%%%%%%%%%%%%%%%%%%%%%%%%%%%%%%%%%%%%%%%%%%%%%%%%%%%%
\section{\tc{black}{Worst-Case Load-Shedding Problem}}
\label{sec.loadshed}

In this section, we formulate the worset-case load-shedding problem for electrical power grids with AC power flow models. In contrast to existing models in literature that describe AC power flow between individual transmission lines, we take advantage of the incidence matrix to encode the network connection in a matrix form. The compact representation of the model facilitates the derivation of the first-order derivatives and enables the convergence analysis in subsequent sections.

\tc{black}{Following~\cite{donloplespinyanmez08,pinmezdonles10}, we consider a lossless power network with $n$ buses and $m$ lines.} A line $l$ connecting bus $i$ and bus $j$ can be described by a vector $e_l \in \R^n$ with $1$ and $-1$ at the $i$th and $j$th elements, respectively, and $0$ everywhere else. Let $E = [ \, e_1 \cdots e_m \, ] \in \R^{n \times m}$ be the incidence matrix that describes $m$ transmission lines of the network, and let $D \in \R^{m \times m}$ be the diagonal matrix with the $l$th diagonal element being the admittance of line $l$. For a lossless power network with fixed voltage at the buses, the active AC power flow equation can be written in a vector form~\cite{donloplespinyanmez08,pinmezdonles10}
\begin{equation}
\label{eq.ac}
E D \sin (E^T \theta) \,=\, P,
\end{equation}
where $\theta \in \R^{n}$ is the phase angles and $P \in \R^n$ is the real power injection at the buses. \tc{black}{Reactive power equation over networks can be written similarly in a vector form~\cite{donloplespinyanmez08}. One can extend this model to include per-unit voltages of buses; see~\cite{donloplespinyanmez08,pinmezdonles10} for detail.}

We enumerate the buses such that the power injection $P$ can be partitioned into a load vector $P_d \leq 0$ and a generation vector $P_g > 0$, thus, $P = [P_d^T \; P_g^T]^T$. The sequence of buses indexed in $P$ is the same as that of the columns of the incidence matrix $E$. Since the power system is lossless, the sum of load is equal to the sum of generation
\[
  \bfo^T P \,=\, 0,
\]
where $\bfo$ is the vector of all ones. 

Let $\gamma \in \{0,1\}^m$ denote whether a line is in service or not: $\gamma_l = 1$ if line $l$ is in service and $\gamma_l = 0$ if line $l$ is out of service. Let $z = [z_d^T \; z_g^T]^T \in \R^n$, where $z_d \geq 0$ and $z_g \leq 0 $ are the load-shedding vector and the generation reduction vector, respectively. It follows that
\[
P_d \,\leq \,P_d \,+\, z_d \,\leq\, 0, 
\]
where the upper  bound $0$ enforces $P_d + z_d$ to be a load vector. Similarly, we have
\[
0 \,\leq\, P_g \,+\, z_g \,\leq\, P_g,
\]
where the lower bound $0$ enforces $P_g + z_g$ to be a generator vector. Since the load shed must be equal to the generation reduction, we have
\[
  \bfo^T z \,=\, 0.
\]
The active power flow equation with possible line removal can be written as
\[
E D \diag(\gamma) \sin(E^T \theta) \,=\, P \,+\, z,
\]
where $\diag(\gamma)$ is a diagonal matrix with its main diagonal equal to $\gamma$. 

\tc{black}{Our objective is to identify a small number of lines in the AC-model power network whose removal results in the maximum load shedding.} Thus, we consider the following {\em worst-case load-shedding\/} problem:
\begin{subequations}
  \label{eq.loadshed}
   \begin{alignat}{2}
     \dps \maxi_{\gamma, \, \theta, \, z} \quad & \mbox{LoadShedding} \,=\,  \bfo^T z_d && \\
     \st                 \quad & E D \diag(\gamma) \sin(E^T \theta) \,=\, P \,+\, z && \label{eq.powerflow}\\[0.1cm]
                         %& \diag(\gamma) \,=\, \diag(\gamma)  && \label{eq.Gamma} \\[0.1cm]
                         & \gamma \in \{0, 1\}^m, \quad m - \bfo^T \gamma \, = \, K && \label{eq.gamma} \\[0.1cm]
                         & \bfo^T z \,=\, 0, \quad z \,=\, [\, z_d^T \; z_g^T \,]^T && \label{eq.z} \\[0.1cm]
                         & 0 \leq z_d \leq -P_d, \quad
                           -P_g \leq z_g \leq 0 && \label{eq.zbound} \\[0.1cm]
                         & -\frac{\pi}{2} \,\leq\, E^T \theta \,\leq\, \frac{\pi}{2}. && \label{eq.theta}
   \end{alignat}
\end{subequations}
	The decision variables are the phase angle $\theta$, the reduction of load $z_d$, the reduction of generation $z_g$, and the out-of-service line indicator $\gamma$. The problem data are the incidence matrix $E$ for the network topology, the admittance matrix $D$ for the transmission lines, the real power injection $P$ at the buses, and the number of out-of-service lines $K$. 
	
Our load-shedding problem is based on the model introduced in~\cite{donloplespinyanmez05}. Related models have been employed for the continguency analysis in~\cite{donloplespinyanmez08} and vulnerability analysis in~\cite{pinmezdonles10}. \tc{black}{In particular, the AC model in~\cite{donloplespinyanmez08} includes both active and reactive power flow equations with varying voltage magnitudes. In this paper, we focus on the active power flow equation with fixed voltages as a step towards addressing the load-shedding problem with the full AC power flow model. Note that the angle difference between the buses $E^T \theta$ takes values between $-\pi/2$ and $\pi/2$. This is in contrast to the assumption of small angle differences employed in DC power flow models~\cite{linche17}.} 

\tc{black}{While we assume a lossless network, the lossless constraint $\bfo^T z = 0$ can be extended to $\bfo^T z  \leq 0$ that takes into account loss over transmission. Similarly, the constraint on power generation $z_g \leq 0$ can be replaced by $z_g \leq \bar{P}_g$ where $\bar{P}_g > 0$. This allows increase in the power generation for re-dispatch flexibility of generators. These extensions can be accommodated in the proposed approach in subsequent sections.}

%We test the loadshedding problem using BONMIN~\cite{bonmin} and SCIP~\cite{scip} on the IEEE 118-bus test case. These solvers aim at finding global optimal solutions via branch-and-bound. Our numerical experiment indicates that the MINLP problem for a medium size power systems (e.g., IEEE 118-bus case) is beyond the capability of state-of-the-art MINLP solvers. This motivates us to develop an alternative approach in the subsequent sections.

% \begin{equation}
%    \begin{array}{ll}
%      \dps \max_{\gamma, \theta, \, z} \quad & {\rm LoadShedding}(\theta,z) \,=\,  \bfo^T z_d  \\
%      \st                 \quad & E D \diag(\gamma) \sin (E^T \theta) \,=\, P \,+\, z \\[0.1cm]
%                          & \diag(\gamma) \,=\, \diag(\gamma)  \\[0.1cm]
%                          & \gamma \in \{0, 1\}^m, \quad m - \bfo^T \gamma \, = \, K \\[0.1cm]
%                          & \bfo^T z \,=\, 0, \quad z \,=\, [\, z_d^T \; z_g^T \,]^T \\[0.1cm]
%                          & P_d \leq P_d + z_d \leq 0 \\[0.1cm]
%                          & 0 \leq P_g + z_g \leq P_g  \\[0.1cm]
%                          & -\dfrac{\pi}{2} \,\leq\, E^T \theta \,\leq\, \dfrac{\pi}{2}. 
%    \end{array}
% \end{equation}

% This max-min problem can be interpreted as an attacker-defender model. The attacker intends to cause the most severe demage by taking a fixed number of $K$ lines out of service. Given an attack scenario, the defender tries to minimize the load-shedding subject to power flow equation and other engineering constraints.
% This bilevel MINLP is hard. We will consider an alternative formulation. 

\section{Separable Structure}
\label{sec.separable}

The worst-case load-shedding problem contains nonlinear constraints and binary variables. One source of nonlinearity is the sinusoidal function and another source is the multiplication between $\diag(\gamma)$ and $\sin(E^T \theta)$. Therefore, it falls into the class of mixed-integer nonlinear programs~(MINLPs), which are very challenging problems. In particular, finding a feasible point for MINLPs can be computationally expensive or even NP-hard~\cite{fleley94,ley01,abhleylin10}.

The maximum load-shedding problem~\eqref{eq.loadshed} turns out to have a separable structure that can be exploited. In what follows, we discuss this structure and develop an algorithm based on the proximal alternating linearization method.

A closer look at~(\ref{eq.loadshed}) reveals that the only constraint that couples all decision variables, $\theta$, $z$, and $\gamma$, is the AC power flow equation~(\ref{eq.powerflow}). Otherwise, the binary variable, $\gamma$, is subject only to the cardinality constraint~(\ref{eq.gamma}). The load-shedding and the generation-reduction variables $z_l$, $z_g$ are subject to the losslessness constraint~(\ref{eq.z}) and the box constraint~(\ref{eq.zbound}). The phase angles of the buses, $\theta$, are subject only to the linear inequality constraint~(\ref{eq.theta}). Therefore, the constraints in the load-shedding problem~(\ref{eq.loadshed}) are separable with respect to $\theta$, $z$, and $\gamma$, provided that the power flow equation~(\ref{eq.powerflow}) is relaxed. 

We next penalize the error in the power flow equation~(\ref{eq.powerflow}) and include the penalty in the cost function. Let us denote the coupling constraint as
\[
c(\gamma,z,\theta) \,=\, E D \diag(\gamma) \sin(E^T \theta) \,-\, (P \,+\, z)
\]
and consider %the minimization of the partial augmented Lagrangian function of~\eqref{eq.loadshed}:
\begin{equation}
  \label{eq.lsp}
   \begin{array}{ll}
     \dps \mini_{\gamma, \, z, \, \theta}  & H_\rho(\gamma,z,\theta) \,\DefinedAs\,  - \bfo^T z_d   \,+\, \dfrac{\rho}{2} \| c(\gamma,z,\theta) \|_2^2 \\[0.2cm]
  \st & \eqref{eq.gamma}, \eqref{eq.z}, \eqref{eq.zbound}, \eqref{eq.theta},
   \end{array}
\end{equation}
where $\rho$ is a positive coefficient. Clearly,  (\ref{eq.lsp}) is a relaxation of the worst-case load-shedding problem~(\ref{eq.loadshed}), since the power flow equation \[c(\gamma,z,\theta) \;=\; 0\] is no longer enforced. \tc{black}{Note that we minimize the negative of load shedding and we follow the convention of minimizing the constraint violation.} The penalty of the constraint violation is controlled by the positive scalar $\rho$. By solving the relaxed problem~(\ref{eq.lsp}) with a sufficiently large $\rho$, the solution of~(\ref{eq.lsp}) converges to the solution of~(\ref{eq.loadshed}). Additional background on penalty methods can be found in~\cite[Chapter 13]{lueye08}.

\section{Proximal Alternating Linearization Method}
\label{sec.PALM}

%The ADMM algorithm has proved effective for many problems; see the survey paper~\cite{boyparchupeleck11}. Among other applications in control, ADMM has been successfully applied to the design of sparse feedback gain~\cite{linfarjov13}, the leader selection problem in consensus networks~\cite{linfarjov14}, and the completion of state covariance~\cite{linjovgeo13}.

In this section, we develop a proximal alternating linearization method~(PALM) that exploits the separable structure of the worst-case load-shedding problem. Roughly speaking, PALM minimizes the cost function by cycling through variables while keeping other variables fixed. The original problem is thus broken down into a sequence of partial problems that are more amenable to efficient algorithms or even closed-form solutions. 

We begin by introducing the following indicator functions of the constraint sets:
\begin{equation}
  \label{eq.phi1}
  \phi_1(\gamma) =
  \left\{
    \begin{array}{ll}
      0, & \mbox{if~} \gamma \,\in \,\{0,1\}^m \mbox{~and~} m - \bfo^T \gamma \,=\, K  \\[0.1cm]
      \infty, & \mbox{otherwise},
    \end{array}
  \right.
\end{equation}
\begin{equation}
  \label{eq.phi2}
  \phi_2(z) =
  \left\{
    \begin{array}{ll}
      0, & \mbox{if~} \quad {\bf 0} \,\leq\,  z_d \,\leq\, - \, P_d \\[0.1cm]
         & \mbox{and~} - P_g \,\leq \, z_g \,\leq \, {\bf 0} \\[0.1cm]
         & \mbox{and~} \bfo^T z \, = \, 0 \\[0.1cm]
      \infty, & \mbox{otherwise},
    \end{array}
  \right.
\end{equation}
and
\begin{equation}
  \label{eq.phi3}
  \phi_3(\theta) =
  \left\{
    \begin{array}{ll}
      0, & \mbox{if~} -\dfrac{\pi}{2} \, \leq \, E^T \theta \,\leq \, \dfrac{\pi}{2} \\[0.2cm]
      \infty, & \mbox{otherwise}.
    \end{array}
  \right.
\end{equation}
With these indicator functions, the minimization problem~(\ref{eq.lsp}) can be compactly expressed as
\begin{equation}
  \label{eq.phi}
  \begin{array}{ll}
\dps \mini_{\gamma, \, z, \, \theta} &
\Phi(\gamma,z,\theta)
\, = \,
 \phi_1(\gamma) + \phi_2(z) + \phi_3(\theta) \\
& \qquad \qquad \quad \,+\, H_\rho(\gamma,z,\theta).
  \end{array}
\end{equation}

%A similar approach based on the alternating direction method of multipliers is proposed in~\cite{linche16} for the load shedding problem with nonlinear AC power flow equations. It is observed that ADMM may cycle without convergence when the coefficient for the quadratic penalty term is not sufficiently large~\cite{linche16}. In contrast, we show in Section~\ref{sec.conv} that the proximal alternating method is guaranteed to globally converge to a stationary point of~\eqref{eq.loadshed}. 

% We define the proximal operator associated with a function $\phi$ as
% \[
%   \prox_\lambda^\phi(x) \, \DefinedAs \, \amin_u \left\{ \phi(u) \,+\, \frac{\lambda}{2} \| u - x \|^2  \right\}
% \]
% where $\lambda$ is a positive coefficient. 
% \[
%   \begin{array}{l}
%     \gamma^{k+1} \in \prox_{a_k}^{\phi_1} \left( \gamma^k - \frac{1}{a_k} \nabla_\gamma H_\rho (\gamma^k,z^k,\theta^k) \right) \\
%     z^{k+1} \in \prox_{b_k}^{\phi_2} \left( z^k - \frac{1}{b_k} \nabla_z H_\rho (\gamma^{k+1},z^k,\theta^k) \right)  \\
%     \theta^{k+1} \in \prox_{c_k}^{\phi_3} \left( \theta^k - \frac{1}{c_k} \nabla_\theta H_\rho (\gamma^{k+1},z^{k+1},\theta^k) \right)
%   \end{array}
% \]

The PALM algorithm uses the following iterations 
\begin{subequations}
  \label{eq.palm}
  \begin{align}
\label{eq.palm.gam}
    \gamma^{k+1} &\in \amin_\gamma \left\{ \phi_1(\gamma) \,+\, \frac{a_k}{2} \| \gamma - u^k \|^2_2 \right\} \\
\label{eq.palm.z}
    z^{k+1} &\in \amin_z \left\{ \phi_2(z) \,+\, \frac{b_k}{2} \| z - v^k \|^2_2 \right\} \\
\label{eq.palm.theta}
    \theta^{k+1} &\in \amin_\theta \left\{ \phi_3(\theta) \,+\, \frac{c_k}{2} \| \theta - w^k \|^2_2 \right\} ,
  \end{align}
\end{subequations}
where $a_k$, $b_k$, and $c_k$ are positive coefficients. In other words, PALM minimizes $\Phi$ with respect to $\gamma$, $z$, and $\theta$, one at a time, while fixing the other variables constant. The quadratic proximal terms penalize the deviation of decision variables $(\gamma,z,\theta)$ from $(u^k,v^k,w^k)$ 
\begin{equation}
  \label{eq.uvw}
  \begin{array}{l}
  u^k \,=\,  \gamma^k - \frac{1}{a_k} \nabla_\gamma H_\rho (\gamma^k,z^k,\theta^k) \\[0.2cm]
  v^k \,=\,  z^k - \frac{1}{b_k} \nabla_z H_\rho (\gamma^{k+1},z^k,\theta^k) \\[0.2cm]
  w^k \,=\,  \theta^k - \frac{1}{c_k} \nabla_\theta H_\rho (\gamma^{k+1},z^{k+1},\theta^k).
  \end{array}
\end{equation}
\tc{black}{Note that  $(u^k,v^k,w^k)$ is a linear combination of $(\gamma^k,z^k,\theta^k)$ and the corresponding partial gradient of $(\nabla_\gamma H_\rho,\nabla_z H_\rho,\nabla_\theta H_\rho)$, hence the term linearization in PALM.} We refer to~\cite{parboy14} for extensive discussions on the proximal algorithms and~\cite{bolsabteb14} for the generic PALM algorithms.

\subsection{Efficient Solutions to Subproblems}
\tc{black}{The minimization problems~\eqref{eq.palm} are projections on the corresponding constraint sets in~\eqref{eq.phi1}-\eqref{eq.phi3}. In particular, the projection on the convex sets~\eqref{eq.phi2}-\eqref{eq.phi3} can be computed efficiently. For the projection on the nonconvex set~\eqref{eq.phi1}, it turns out that the solution has a closed-form expression.}

We begin with the projection on the convex sets. The $z$-minimization problem~\eqref{eq.palm.z} can be expressed as 
\begin{equation}
  \label{eq.z-min}
  \begin{array}{ll}
    \mini & \dfrac{b_k}{2} \| z - v^k \|_2^2  \\[0.25cm]
    \st         & L \,\leq \, z \,\leq\, U, \quad \bfo^T z \,=\,0,
  \end{array}
\end{equation}
where the lower bound is $L = - [{\bf 0}^T \, P_g^T]^T$ and the upper bound is $U = - [P_d^T \, {\bf 0}^T]^T$. The solution of this convex quadratic program with box constraints and a {\em single\/} equality constraint, $\bfo^T z = 0$, can be computed efficiently.  

The $\theta$-minimization problem~\eqref{eq.palm.theta} can be expressed as 
\begin{equation}
  \label{eq.theta-min}
  \begin{array}{ll}
    \mini       & 
                   \dfrac{c_k}{2} \, \| \theta - w^k \|_2^2 \\[0.25cm]
    \st         & - \dfrac{\pi}{2} \,\leq\, E^T \theta \,\leq\, \dfrac{\pi}{2}.
  \end{array}
\end{equation}
This bound-constrained least-squares problem can be solved efficiently.

We next provide a closed-form solution to the $\gamma$-minimization problem~\eqref{eq.palm.gam}
\begin{equation}
  \label{eq.gamma-min}
  \begin{array}{ll}
    \mini       &  \dfrac{a_k}{2} \| \gamma - u^k \|_2^2 \\[0.25cm]
    \st         & \gamma \,\in\, \{0,1\}^m, \quad   \bfo^T \gamma \,=\, m - K.
  \end{array}  
\end{equation}
\begin{lemma}
\label{lem.gamsol}
Let $[u^k]_K$ be the $K$th smallest element of $u^k$. The $i$th element of the solution to~\eqref{eq.gamma-min} is given by
\begin{equation}
\label{eq.gamsol}
\gamma_i
\,=\, 
\left\{
\begin{array}{ll}
1 & \mbox{if } u_i^k \,\geq\, [u^k]_K \\
0 & \mbox{otherwise} ,
\end{array}
\right.
\end{equation}
for $i=1,\ldots,m$.
\end{lemma}
The proof can be found in Appendix~\ref{pro.gamsol}.

\tc{black}{Proximal algorithms typically rely on convexity assumptions to guarantee convergence~\cite{parboy14}. In contrast, the PALM algorithm does not require the objective or the constraints to be convex. PALM relies on the smoothness condition of the coupling term $H_\rho$ and the Lipschitz conditions of the partial gradients $\nabla H_\rho$. Another feature of PALM is that it does not require stepsize rules as in typical descent-based methods. This is because the Lipschitz conditions guarantee the descent of the objective value in each PALM iteration; see Section~\ref{sec.conv}.}

To complete the PALM algorithm, we provide the expressions for $\nabla H_\rho$ and discuss the choice of $a_k$, $b_k$, and $c_k$  in~\eqref{eq.palm}. 
\begin{lemma}
\label{lem.gradH}
The partial gradients $\nabla H_\rho$ with respect to $\gamma$, $z$, and $\theta$ are given by
\begin{subequations}
    \label{eq.gradH}
  \begin{align}
    \nonumber
    \nabla_\gamma H_\rho
    & =  \rho \left( (DE^T E D) \circ (\sin(E^T \theta) \sin(E^T \theta)^T) \right) \gamma \\ 
    \label{eq.gradHgam}
    & -  \rho \left( (\sin(E^T \theta) (P+z)^T E D) \circ I \right) \bfo, \\[0.1cm]
    \nabla_z H_\rho 
    \label{eq.gradHz}
    & = - [\bfo^T \, {\bf 0}^T]^T + \rho ( P + z ) 
     - \rho (E D \Gamma \sin(E^T \theta) )  , \\[0.1cm]
    \nonumber
    \nabla_\theta H_\rho
    & = \rho \, E \diag(\cos(E^T \theta)) \Gamma D E^T \times \\
    & \hspace{0.5in} (E D \Gamma \sin(E^T \theta) - (P + z))
    \label{eq.gradHthe}
  \end{align}
\end{subequations}
where $\circ$ denotes the elementwise product of matrices. 
\end{lemma}
The derivation can be found in Appendix~\ref{sec.gradH}.

The positive coefficients $a_k$, $b_k$, and $c_k$ in~\eqref{eq.palm} and~\eqref{eq.uvw} are determined by 
\[
  \begin{array}{l}
    a_k \,=\, r_1 L_1(z^k,\theta^k) \\
    b_k \,=\, r_2 L_2(\gamma^{k+1},\theta^k) \\
    c_k \,=\, r_3 L_3(\gamma^{k+1},z^{k+1}) ,
  \end{array}
\]
where positive constants $r_i > 1$ for $i=1,2,3$. The Lipschitz constants $L_i$ for the partial gradients $\nabla H_\rho$ are given by
\begin{subequations}
\begin{align}
\label{eq.L1}
&    L_1(z^k,\theta^k) = \rho \, \| (DE^T E D) \circ (\sin(E^T \theta^k) \sin(E^T \theta^{k})^T ) \|  \\
\label{eq.L2}
&    L_2(\gamma^{k+1},\theta^k) = \rho  \\
\label{eq.L3}
&    L_3(\gamma^{k+1},z^{k+1}) = \rho \|E\|^2 ( 2 \|Q^{k+1}\| + \|R^{k+1}\|)
  \end{align}
\end{subequations}
where $\| \cdot \|$ denotes the maximum singular value of a matrix. The derivation of the Lipschitz constants are provided in Section~\ref{sec.conv}. 

We conclude this section by summarizing PALM in Algorithm~\ref{alg.palm}.

\begin{algorithm}[t]
\caption{Proximal alternating linearization method} 
\label{alg.palm}
\begin{algorithmic}
\STATE Start with any $(\gamma^k,z^k,\theta^k)$ and set $k \gets 0$.
\FOR{$k=0,1,2,\ldots$ until convergence}
\STATE // {\bf $\gamma$-minimization:}
\STATE Set $a_k = r_1 L_1(z^k,\theta^k)$ where $r_1 > 1$ and $L_1$ in~\eqref{eq.L1}.
Solve problem~\eqref{eq.gamma-min} via the closed-form expression~\eqref{eq.gamsol}  to get $\gamma^{k+1}$. 
\STATE // {\bf $z$-minimization:}
\STATE Set $b_k = r_2 L_2(\gamma^{k+1},\theta^k)$ where $r_2 > 1$ and $L_2$ in~\eqref{eq.L2}. Solve the convex quadratic problem~\eqref{eq.z-min} to get $z^{k+1}$. 
\STATE // {\bf $\theta$-minimization:}
\STATE Set $c_k = r_3 L_3(\gamma^{k+1},z^{k+1})$ where $r_3 > 1$ and
$L_3$ in~\eqref{eq.L3}. 
Solve the convex quadratic problem~\eqref{eq.theta-min} to get $\theta^{k+1}$.
\STATE Set $(\gamma^k,z^k,\theta^k) \gets (\gamma^{k+1},z^{k+1},\theta^{k+1})$.
\ENDFOR
\end{algorithmic}
\end{algorithm}

\section{Convergence Analysis}
\label{sec.conv}

In this section, we show that Algorithm~\ref{alg.palm} converges to a critical point of the nonconvex problem~\eqref{eq.lsp}. This convergence behavior is independent of the initial guess of the decision variables.  Furthermore, the objective value $\Phi$ is monotonically decreasing with the number of iterates, that is, 
\[
\Phi (\gamma^{k+1},z^{k+1},\theta^{k+1}) \, \leq \, \Phi (\gamma^{k},z^{k},\theta^{k}).
\] 
\tc{black}{This feature of monotonic decreasing allows us to monitor the progress of PALM. It also allows us to check if the implementation is correct in practice.}

We begin with two technical lemmas on the Lipschitz properties of $\Phi$.

\begin{lemma} \label{lem.pro}
  The objective function $\Phi$ in~\eqref{eq.phi} satisfies the following properties:
  \begin{enumerate}
  \item \label{pro.lb} 
   $\inf_{\gamma,z,\theta} \Phi(\gamma,z,\theta) > -\infty$, $\inf_\gamma \phi_1(\gamma) > -\infty$, $\inf_z \phi_2(z) > -\infty$, and $\inf_\theta \phi_3(\theta) > -\infty$.  
  \item \label{pro.lip} 
   For fixed $(z,\theta)$, the partial gradient $\nabla_\gamma H_\rho$ is globally Lipschitz, 
\[
  \begin{array}{l}
  \| \nabla_\gamma H_\rho(\gamma_1,z,\theta) - \nabla_\gamma H_\rho(\gamma_2,z,\theta)  \| \\[0.1cm]
    \hspace{1.5in}
  \, \leq \, L_1(z,\theta) \| \gamma_1 - \gamma_2\| 
  \end{array}
\]
for all $\gamma_1$ and $\gamma_2$. Likewise, for fixed $(\gamma,\theta)$, the partial gradient $\nabla_z H_\rho$ satisfies
\[
  \begin{array}{l}
  \| \nabla_z H_\rho(\gamma,z_1,\theta) - \nabla_z H_\rho(\gamma,z_2,\theta)  \| \\[0.1cm]
    \hspace{1.5in}
  \, \leq \, L_2(\gamma,\theta) \| z_1 - z_2\|   
  \end{array}
\]
for all $z_1$ and $z_2$, and for fixed $(z,\gamma)$, 
\[
  \begin{array}{l}
  \| \nabla_\theta H_\rho(\gamma,z,\theta_1) - \nabla_\theta H_\rho(\gamma,z,\theta_2)  \| \\[0.1cm]
    \hspace{1.5in}
  \, \leq \, L_3(\gamma,z) \| \theta_1 - \theta_2\|    
  \end{array}
\]
for all $\theta_1$ and $\theta_2$.
\item \label{pro.lipbd}
  There exist positive constants $s_1, s_2, s_3$ such that
  \begin{equation}
    \label{eq.lipbd}
    \begin{array}{c}
    \sup_k \{ L_1(z^k,\theta^k) \} \, \leq \, s_1,
      \\[0.1cm]
    \sup_k \{ L_2(\gamma^k,\theta^k)  \} \, \leq \, s_2,
      \\[0.1cm]
    \sup_k \{ L_3(\gamma^k,z^k) \}  \, \leq \, s_3.
    \end{array}
  \end{equation}
\item \label{pro.lipc2} 
  The entire gradient $\nabla H_\rho(\gamma,z,\theta)$ is Lipschitz continuous on bounded subsets of $\R^{m} \times \R^{n} \times \R^n$.
  \end{enumerate}
\end{lemma}

\begin{remark}
\tc{black}{Property~\ref{pro.lb}) is necessary for the minimization problems in Algorithm~\ref{alg.palm}, and thus the minimization of $\Phi$,  to be well defined. Property~\ref{pro.lip}) on the globally Lipschitz bounds is critical for the convergence of PALM. Note that the block Lipschitz property of $\nabla H_\rho$ is weaker than the globally Lipschitz assumption of $\Phi$ in joint variables $(\gamma,z,\theta)$ in standard proximal methods~\cite{bolsabteb14}. Property~\ref{pro.lipbd}) guarantees that the Lipschitz constants for partial gradients are upper bounded by finite numbers. Property~\ref{pro.lipc2}) is a mild condition which holds when $H_\rho$ is twice continuously differentiable.}
\end{remark}

\begin{proof}
Property~\ref{pro.lb}) is a direct consequence of the nonnegativity of $H_\rho$ and the definition of the indicator functions $\phi_1$, $\phi_2$, and $\phi_3$. Property~\ref{pro.lipc2}) holds because $H_\rho$ is twice continuously differentiable.

To show Property~\ref{pro.lip}), recall that for fixed $(z^k,\theta^k)$ the Lipschitz constant $L_1(z^k,\theta^k)$ of $\nabla_\gamma H_\rho$ is determined by
\[
  \begin{array}{l}
  \| \nabla_\gamma H_\rho(\gamma_1,z^k,\theta^k) - \nabla_\gamma H_\rho(\gamma_2,z^k,\theta^k)  \| \\[0.1cm]
    \hspace{1.6in}
  \, \leq \, L_1(z^k,\theta^k) \| \gamma_1 - \gamma_2\| 
  \end{array}
\]
for all $\gamma_1$ and $\gamma_2$. Since $\nabla_\gamma H_\rho$ is an affine function of $\gamma$ (see~\eqref{eq.gradHgam}), it follows that 
\[
   L_1(z^k,\theta^k) = \rho \| (DE^T E D) \circ (\sin(E^T \theta^k) \sin(E^T \theta^{k})^T ) \|.
\]
For fixed $(\gamma^{k+1},\theta^k)$, the Lipschitz constant $L_2(\gamma^{k+1},\theta^k)$ of $\nabla_z H_\rho$ is determined by
\[
  \begin{array}{l}
  \| \nabla_z H_\rho(\gamma^{k+1},z_1,\theta^{k}) - \nabla_z H_\rho(\gamma^{k+1},z_2,\theta^{k})  \| \\[0.1cm]
    \hspace{1.5in}
  \, \leq \, L_2(\gamma^{k+1},\theta^k) \| z_1 - z_2\|   
  \end{array}
\]
for all $z_1$ and $z_2$. Since $\nabla_z H_\rho$ is an affine function of $z$ (see~\eqref{eq.gradHz}), it follows that
\[
L_2(\gamma^{k+1},\theta^k) \,=\, \rho.
\]
For fixed $(\gamma^{k+1},z^{k+1})$, the Lipschitz constant $L_3(\gamma^{k+1},z^{k+1})$ of $\nabla_\theta H_\rho$ is determined by
\[
  \begin{array}{l}
  \| \nabla_\theta H_\rho(\gamma^{k+1},z^{k+1},\theta_1) - \nabla_\theta H_\rho(\gamma^{k+1},z^{k+1},\theta_2)  \| \\[0.1cm]
    \hspace{1.45in}
  \, \leq \, L_3(\gamma^{k+1},z^{k+1}) \| \theta_1 - \theta_2\|    
  \end{array}
\]
for all $\theta_1$ and $\theta_2$. The Lipschitz constant for $\nabla_\theta H_\rho$ is given by~(see Appendix~\ref{sec.lip} for derivation)
\[
L_3(\gamma^{k},z^{k}) =  \rho \|E\|^2 ( 2 \|Q^{k}\| + \|R^{k}\|)
\]
where
\[
  Q^{k} \,=\, \Gamma^{k} D E^T E D \Gamma^{k},
  \quad
  R^{k} \,=\, \Gamma^{k} DE^T (P+z^{k}).
\]

The proof is complete by establishing Property~\ref{pro.lipbd}). Since the maximum singular value of the elementwise product of two matrices is upper bounded by the product of maximum singular values of individual matrices~\cite[Theorem 5.5.1]{horjoh91}, it follows that
\begin{align*}
     L_1(z^k,\theta^k) \leq  \rho \| DE^TED \|  \cdot \| \sin(E^T \theta^k) \sin(E^T \theta^{k})^T \| ,
\end{align*}
thus, $s_1 = \rho m \|ED \|^2$. From~\eqref{eq.L2}, we have $s_2 = \rho$ and from~\eqref{eq.L3}, we have $s_3 = \rho \|E\|^2 \|ED\|^2 (2 + \|P\|)$.
\end{proof}

The convergence of PALM relies on the so-called KL property. We refer to~\cite{kur98,boldanley10,bolsabteb14} for detailed discussions on the KL theory. We next recall a few definitions needed for our PALM algorithm.

\begin{definition}
Let $f: \R^d \to (-\infty,+\infty]$ be proper and lower semicontinuous. The function $f$ is said to have the {\em Kurdyka-Lojasiewicz (KL) property\/} at $\bar{u} \in \mbox{dom} \, \partial f \DefinedAs \{ u \in \R^d : \partial f(u) \neq \emptyset  \}$ if there exist $\eta \in (0, +\infty]$, a neighborhood ${\cal N}$ of $\bar{u}$, and a function $\psi$ such that for all 
\[
  u \in {\cal N} \cap \{ f(\bar{u}) < f(u) < f(\bar{u}) + \eta \},
\]
the following inequality holds:
\[
  \psi'(f(u) - f(\bar{u})) \cdot \mbox{dist}(0,\partial f(u)) \,\geq\, 1,
\]
where $\mbox{dist}(x,s) \DefinedAs \inf \{ \|y-x\| : y \in s \}$ denotes the distance from a point $x \in \R^d$ to a set $s \subset \R^d$. A function $f$ is called a {\em KL function\/} if $f$ satisfies the KL property at each point of dom $\partial f$.
\end{definition}

\tc{black}{The KL property is a technical condition that controls the difference in function value by its gradient. It turns out that a large class of functions that arise in modern applications satisfy the KL property~\cite{kur98,boldanley10,bolsabteb14}. One useful way of establishing the KL property is via the connection with the semi-algebraic sets and the semi-algebraic functions.}

\begin{definition}
A subset ${\cal S}$ of $\R^d$ is a real {\em semi-algebraic set\/} if there exists a finite number of real polynomial functions $g_{ij}$ and $h_{ij} : \R^d \to \R$ such that
\[
{\cal S} \,=\, \bigcup_{j=1}^p \bigcap_{i=1}^q \{  u \in \R^d: g_{ij}(u) = 0 ~\mbox{and}~ h_{ij}(u) < 0  \}.
\]
\end{definition}

\begin{definition}
A function $h:\R^d \to (-\infty,+\infty]$ is called {\em semi-algebraic function\/} if its graph $\{(u,v) \in \R^{d+1} : h(u) = v \}$ is a semi-algebraic subset of $\R^{d+1}$.
\end{definition}

Given these definitions we show the KL property of $\Phi$.
\begin{lemma} \label{lem.kl}
  The objective function $\Phi$ in~\eqref{eq.phi} satisfies the Kurdyka-Lojasiewicz (KL) property.
\end{lemma}

\begin{proof}
\tc{black}{Since analytic functions satisfy the Lojasiewicz inequality~\cite{kur98,boldanley10} and since $H_\rho$ is the multiplication of polynomial function and sinusoidal function, it follows that $H_\rho$ satisfies the KL property.}

\tc{black}{The nonsmooth parts of $\Phi$, namely, the indicator functions $\phi_1$, $\phi_2$, and $\phi_3$, are lower semicontinuous. Since a proper, lower semicontinuous, and semi-algebraic function satisfies the KL property~\cite[ Theorem 3]{bolsabteb14}, it suffices to show that $\phi_1$, $\phi_2$, and $\phi_3$ are semi-algebraic functions.} Because $\phi_2$ and $\phi_3$ are indicator functions of the semi-algebraic sets~\eqref{eq.phi2}-\eqref{eq.phi3}, they are semi-algebraic functions. To show that $\phi_1$ is semi-algebraic, note that the binary constraint $\gamma_i \in \{0,1\}$ can be expressed as a polynomial equation $\gamma_i (\gamma_i - 1)=0$ for $i=1,\ldots,m$. Thus $\{\gamma \,|\, \gamma \in \{0,1\}, m - \bfo^T \gamma = K\}$ is a semi-algebraic set. Therefore the indicator function $\phi_1$ is semi-algebraic, which completes the proof.
\end{proof}

After establishing Lemma~\ref{lem.pro} and Lemma~\ref{lem.kl}, the main convergence results follow from the pioneering work by Bolte et al.~\cite{bolsabteb14}.
\begin{proposition}
  \label{pro.conv}  
Suppose that $\Phi$ is a KL function that satisfies conditions in Lemma~\ref{lem.pro}. Let $x^k = (\gamma^{k}, z^{k}, \theta^{k})$ be a bounded sequence generated by PALM. The following results hold:
\begin{enumerate}
\item The sequence $\{x^k\}$ has finite length, that is, 
\[
  \sum_{k=1}^\infty \|x^{k+1} - x^k\|_2 \, < \, \infty.
\]
\item The sequence $\{x^k\}$ converges to a critical point $x^* = (\gamma^{*}, z^{*}, \theta^{*})$ of $\Phi$.
\item The sequence $\Phi(x^k)$ is nonincreasing,
\[
  \frac{d}{2} \|x^{k+1} - x^k \|_2^2 \, \leq \, \Phi(x^k) - \Phi(x^{k+1}), 
  \quad k \geq 0
\]
where $d$ is positive constant bounded below.
\end{enumerate}
\end{proposition}
\begin{proof}
The finite length property and the convergence to a critical point follow from Theorem 1 in~\cite{bolsabteb14}. The monotonicity of the objective value is obtained from Lemma 3 in~\cite{bolsabteb14}. 
\end{proof}

\section{Numerical Results}
\label{sec.results}

\tc{black}{In this section, we verify the convergence results of PALM and examine its solution quality in two IEEE test cases. 
The first test case, the IEEE 14-bus system, illustrates the convergence behavior of PALM. The second test case, the IEEE 118-bus system, demonstrates the scalability of the algorithm.}

\subsection{IEEE 14-bus Test Case}

 \begin{figure}
   \centering
   \includegraphics[width=0.48\textwidth]{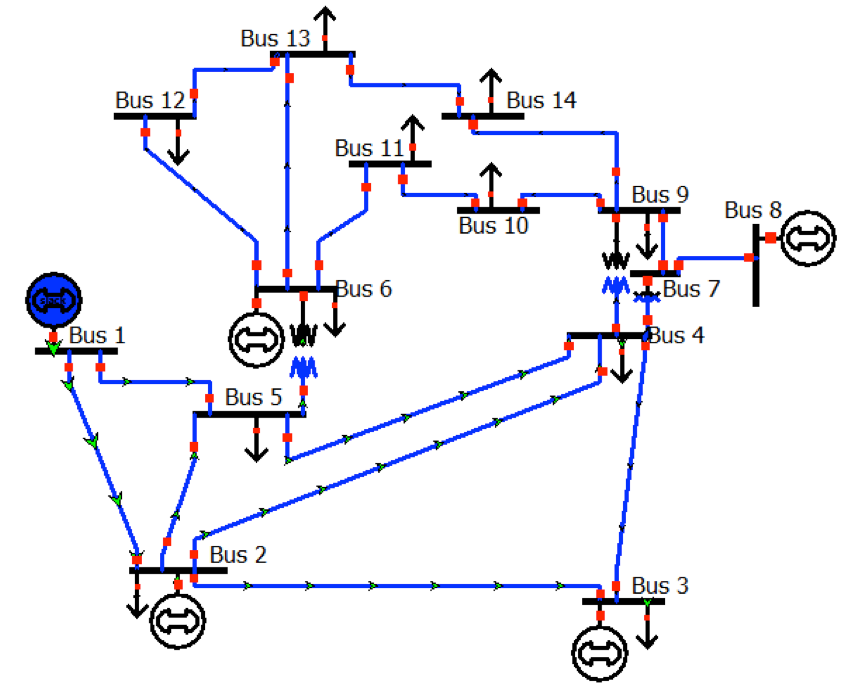}
   \caption{\tc{black}{Diagram of the IEEE 14-bus test case.}}
   \label{fig.IEEE14}
 \end{figure}

\begin{figure*}
  \centering
  \begin{tikzpicture}
    \begin{axis} [width=0.35\textwidth,
      xlabel = PALM iteration index $k$,
      mark repeat={50},
      title = Objective value \mbox{$H_{\rho}(\gamma,z,\theta)$}
      ]
      \addplot table[x=IterNo,y=ObjVal] {Bus14_k5.txt};
    \end{axis}
  \end{tikzpicture}
  \begin{tikzpicture}
    \begin{semilogyaxis} [width=0.35\textwidth,
      xlabel = PALM iteration index $k$,
      mark repeat={50},
      title = Dual residuals
      ]
      \addplot table[x=IterNo,y=theta_res] {Bus14_k5.txt};
      \addplot table[x=IterNo,y=z_res] {Bus14_k5.txt};
      \addplot table[x=IterNo,y=gam_res] {Bus14_k5.txt};
      \legend{$\|\theta^{k+1}-\theta^k\|$,$\|z^{k+1} - z^k\|$,$\|\gamma^{k+1}-\gamma^k\|$}
    \end{semilogyaxis}
  \end{tikzpicture}
  \begin{tikzpicture}
    \begin{semilogyaxis} [width=0.35\textwidth,
      xlabel = PALM iteration index $k$,
      mark repeat={50},
      title = Primal residual \mbox{$\|c(\gamma,z,\theta)\|$}
      ]
      \addplot table[x=IterNo,y=prim_res] {Bus14_k5.txt};
    \end{semilogyaxis}
  \end{tikzpicture}
  \caption{\tc{black}{Convergence results of PALM for the IEEE 14-bus test case: the objective value (left), the dual residuals (middle), and the primal residual (right). The markers show at every 50 iterates.}}
  \label{fig.bus14-k5}
\end{figure*}

 \begin{table}
   \centering
   \caption{\tc{black}{Load-shedding strategy for the IEEE 14-bus test case.}}
   \begin{tabular}{|c|c|c|c|}
     \hline
     $K$ & Load Shed  & Percentage & Lines Removed \\
     \hline
      1  &  80.2   MW   & 18.3\%         & 13  \\
      2  &  90.5  MW   & 20.7\%         & 3, 13  \\
      3  & 105.1  MW   & 24.0\%   & 3, 13, 15 \\
      4  & 188.2  MW   & 43.0\%   &      3, 11, 13, 15 \\
      5  & 285.5       MW   &      65.3\%   &      3, 11, 12, 13, 15\\
     \hline   
   \end{tabular}
   \label{tab.IEEE14}
 \end{table}

%
% \begin{table}
%   \centering
%   \caption{Load-shedding strategy for the IEEE 14-bus test case.}
%   \begin{tabular}{|c|c|c|c|}
%     \hline
%     $K$ & Load Shed  & Percentage & Lines Removed \\
%     \hline
%      1  &  47.3135            MW   & 10.81 \%         & 12  \\
%      2  &  51.3888           MW   & 11.75  \%         & 12, 2 \\
%      3  & 78.6742         MW   & 17.98 \%   & 12, 2, 9 \\
%      4  & 151.4783         MW   & 34.62 \%   & 12, 2, 9, 5 \\
%      5  & 221.3116       MW   & 50.59 \%   & 12, 2, 9, 5, 15\\
%     \hline   
%   \end{tabular}
%   \label{tab.IEEE14}
% \end{table}

%
% \begin{table}
%   \centering
%   \caption{The set of out-of-service lines and the bus types.}
%   \begin{tabular}{|c|c|c|c|c|}
%     \hline
%     Line & Bus & Type & Bus & Type   \\ 
%     \hline
%      12  & 6  & load &  12 & load   \\
%     2   & 1  & gen &  5 & load   \\
%     9   & 4 & load  &  9 & load  \\
%    5   & 2  & gen &  5 & load   \\
%    15   & 7  & load & 9 & load \\
%     \hline
%   \end{tabular}
%   \label{tab.cutsIEEE14}
% \end{table}

% \begin{figure}
%   \centering
%   \includegraphics[width=0.48\textwidth]{IEEE14_cut_nolines}
%   \caption{Diagram of the IEEE 14-bus test case with 5 out-of-service.}
%   \label{fig.IEEE14_cut}
% \end{figure}

Consider the IEEE 14-bus test case shown in Fig.~\ref{fig.IEEE14}. This small system has 5 generator buses, 9 load buses, and 20 transmission lines. We compute the generation profile, $P_g$, and the load profile, $P_d$, by solving the steady-state power flow equations via MATPOWER~\cite{zimburcartho11}. 

\tc{black}{We take out up to 5 lines to track the progress of the worst-case load shedding in this small network. As the out-of-service number of lines increases from $K=1$ to $K=5$, the amount of load shed increases from $18.3\%$ to $65.3\%$ of the total power load; see Table~\ref{tab.IEEE14}.} \tc{black}{It turns out that the set of lines to be taken out of service is a subset of the lines as $K$ increases. This implies the consistency in the set of critical transmission lines for load-shedding.  The out-of-service lines are highlighted in Fig.~\ref{fig.IEEE14}. It is worth mentioning that PALM is initialized with $(\gamma=\bfo,\theta={\bf 0},z={\bf 0})$ for all $K=1,\ldots,5$. In other words, the algorithm starts with full service lines and zero load shed.}

\tc{black}{Figure~\ref{fig.bus14-k5} shows the convergence results of PALM when $5$ lines are removed. The objective function decreases monotonically with the PALM iterations, thereby confirming the prediction in Proposition~\ref{pro.conv}. Furthermore, both the dual residuals and the primal residual decrease monotonically. The fastest convergence of PALM is in the first 200-300 iterations, in this case.} The convergence rate depends on the size of the problem and the choice of parameter $\rho$. While a bigger $\rho$ improves the primal convergence rate, it slows down the dual convergence rate. In practice, we find that $\rho \in [10^4,10^6]$ achieves a good balance between the primal and dual residuals.

\tc{black}{Since we relax the constraint $c(\gamma,z,\theta) = 0$ in~\eqref{eq.lsp}, we check the solution quality in satisfying the power flow equation. As shown in Fig.~\ref{fig.bus14-k5}, the primal residual $\|c(\gamma,z,\theta)\|$ is monotonically decreasing with PALM iterations; in particular, we have $\|c(\gamma,z,\theta)\| \leq 3.5 \times 10^{-3}$ after 1000 iterations. As discussed above, one can further reduce the primal residual by increasing the penalty parameter $\rho$.}

%Techniques that update $\rho$ based on primal and dual residuals can be found in~\cite{boyparchupeleck11}.

%The out-of-service lines are highlighted in Fig.~\ref{fig.IEEE14}. To gain some physical insight into the out-of-service lines, we consider the types of buses with which the lines connect. Table~\ref{tab.cutsIEEE14} shows a total of 5 lines and 8 buses, among which 6 buses are load buses and 2 buses are generator buses.

\subsection{IEEE 118-bus Test Case}

 \begin{figure}
   \centering
   \includegraphics[width=0.48\textwidth]{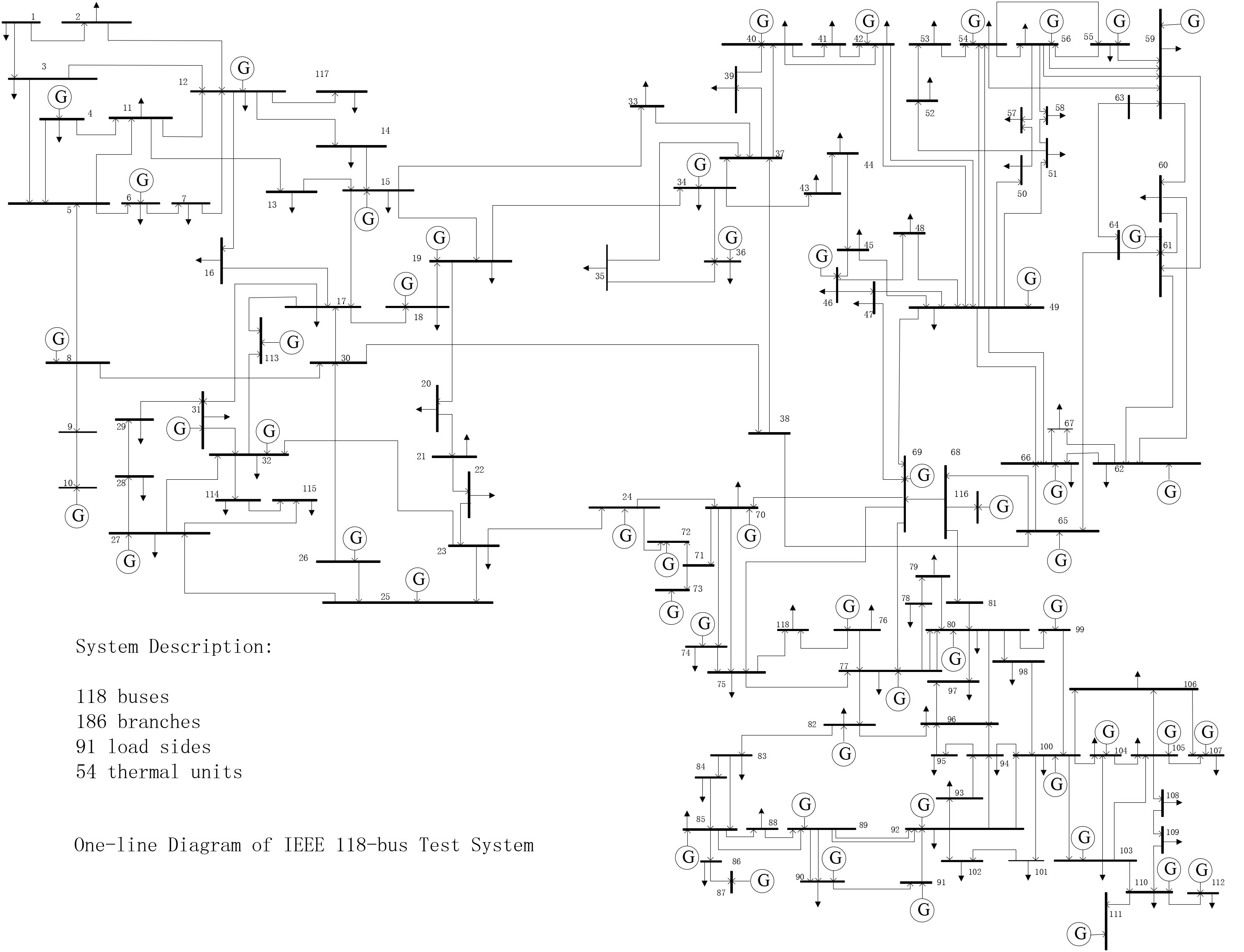}
   \caption{Diagram of the IEEE 118-bus test case.}
   \label{fig.IEEE118}
 \end{figure}

\begin{figure*}
  \centering
  \begin{tikzpicture}
    \begin{axis} [width=0.35\textwidth,
      xlabel = PALM iteration index $k$,
      mark repeat={100},
      title = Objective value \mbox{$H_{\rho}(\gamma,z,\theta)$}
      ]
      \addplot table[x=IterNo,y=ObjVal] {Bus118_k5.txt};
    \end{axis}
  \end{tikzpicture}
  \begin{tikzpicture}
    \begin{semilogyaxis} [width=0.35\textwidth,
      xlabel = PALM iteration index $k$,
      mark repeat={100},
      title = Dual residuals
      ]
      \addplot table[x=IterNo,y=theta_res] {Bus118_k5.txt};
      \addplot table[x=IterNo,y=z_res] {Bus118_k5.txt};
      \addplot table[x=IterNo,y=gam_res] {Bus118_k5.txt};
      \legend{$\|\theta^{k+1}-\theta^k\|$,$\|z^{k+1} - z^k\|$,$\|\gamma^{k+1}-\gamma^k\|$}
    \end{semilogyaxis}
  \end{tikzpicture}
  \begin{tikzpicture}
    \begin{axis} [width=0.35\textwidth,
      xlabel = PALM iteration index $k$,
      mark repeat={100},
      title = Primal residual \mbox{$\|c(\gamma,z,\theta)\|$}
      ]
      \addplot table[x=IterNo,y=prim_res] {Bus118_k5.txt};
    \end{axis}
  \end{tikzpicture}
  \caption{Convergence results of PALM for the IEEE 118-bus test case: the objective value (left), the dual residuals (middle), and the primal residual (right). The markers show at every 100 iterations.}
  \label{fig.bus118-k5}
\end{figure*}
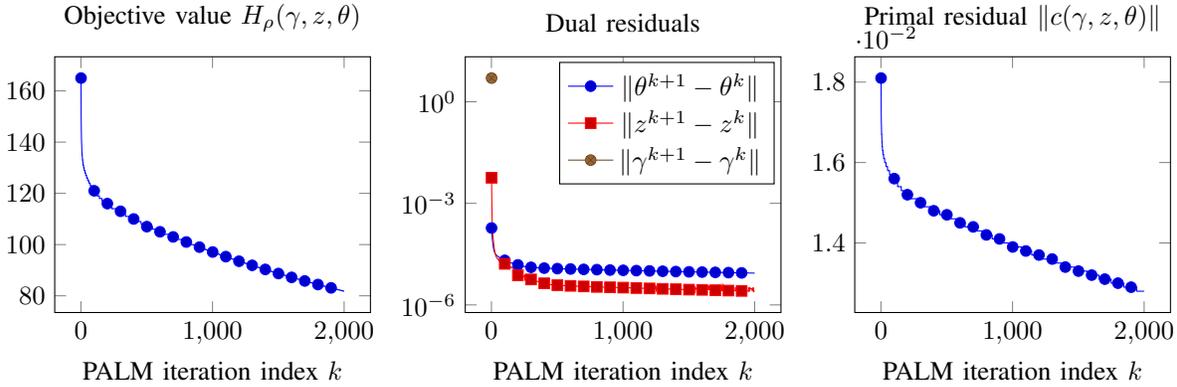

We next consider the IEEE 118-bus test case as shown in Fig.~\ref{fig.IEEE118}. This large power system has 54 generator buses, 64 load buses, and 186 transmission lines. As in the IEEE-14 bus system, the generation profile, $P_g$, and load profile, $P_d$, are obtained by solving the steady-state power flow equations via MATPOWER~\cite{zimburcartho11}. 

While the 118-bus system is much larger than the 14-bus system, the convergence behavior of PALM is quite similar. The objective value, the dual residuals, and the primal residual all decrease monotonically, as shown in Fig.~\ref{fig.bus118-k5}. \tc{black}{After 2000 iterations, the primal residual is smaller than $1.3 \times 10^{-2}$ and the dual residual is smaller than $1.2\times 10^{-5}$. The solution quality is determined by the primal residual, which is $\| c(\gamma,z,\theta) \| \leq 1.3  \times 10^{-2}$ after 2000 PALM iterations. It is worth mentioning that the computational time is less than 10 minutes on a laptop with 8GB RAM running 2.4GHz CPU.}

\tc{black}{Table~\ref{tab.IEEE118} shows the worst-case load-shedding scenarios with removal up to $5$ transmission lines. As observed in 14-bus test case, the most critical lines to be taken out of service form a subset of lines as $K$ increases. For this large system, the load shed percentage is less than $10.2\%$ when $5$ lines are taken out. This is in contrast to the 14-bus system, in which the load shed percentage is more than $65\%$ when $K=5$.}

\tc{black}{To gain some insight into the out-of-service lines, we consider the types of buses with which the lines connect. As shown in Table~\ref{tab.cutsIEEE118}, all critical lines connect the same types of buses, that is, generator to generator and load to load buses. In particular, 4 out of the 5 critical lines connect generator  buses. This indicates the importance of lines between generator buses in the IEEE-118 system.}

%We next consider the phase angle across transmission lines. A larger phase angle implies a larger power transfer over a transmission line. Thus, it is a good indicator of overload in contingency analysis of power grids~\cite{donloplespinyanmez05}. Figure~\ref{fig.IEEE118_phaseangle} shows the phase angle difference $E^T \theta$ across all 186 lines when lines $\{43,   163,   173,   176,   177\}$ are taken out of services. The maximum angle difference is more than $5.47$ degrees. In particular, $85\%$ of the phase angles of are less than 0.5 degrees, indicating a well-balanced operating status.

 \begin{table}
   \centering
   \caption{\tc{black}{Load-shedding strategy for the IEEE 118-bus test case.}}
   \begin{tabular}{|c|c|c|c|}
     \hline
     $K$ & Load Shed  & Percentage & Lines Removed \\
     \hline
      1  &  136.6 MW   &    3.1\%            & 176  \\
      2  &  238.2 MW   &    5.4\%          &    173,   176   \\
      3  &  307.1 MW   &    7.0\%   &     173,   176,   177 \\
      4  &  321.1 MW   &    7.3\%   &     163,   173,   176,   177\\
      5  &  444.0 MW   &    10.1\%   &     43,   163,   173,   176,   177\\
     \hline   
   \end{tabular}
   \label{tab.IEEE118}
 \end{table}

 \begin{table}
   \centering
   \caption{\tc{black}{The set of out-of-service lines and the bus types for the IEEE 118-bus test case.}}
   \begin{tabular}{|c|c|c|c|c|}
     \hline
     Line & Bus & Type & Bus & Type   \\ 
     \hline
      43  & 27  & generator &  32 & generator   \\
     163   & 100  & generator &  103 & generator   \\
     173   & 108 & load  &  109 & load  \\
    176  & 110  & generator &  111 & generator   \\
    177   & 110  & generator & 112 & generator \\
     \hline
   \end{tabular}
   \label{tab.cutsIEEE118}
 \end{table}

%
% \begin{figure}
%   \centering
%   \includegraphics[width=0.48\textwidth]{IEEE118_k5_PhaseAngle.png}
%   \caption{The phase angle across transmission lines after $5$ lines,  $\{174, 119, 110, 135, 120\}$, are taken out of service for the 118-bus system.}
%   \label{fig.IEEE118_phaseangle}
% \end{figure}

\section{Conclusions}
\label{sec.conclusion}
We formulate the worst-case load-shedding problem in AC power networks. We show that this nonconvex control problem has a separable structure that can be exploited by PALM. The PALM algorithm decomposes load-shedding problem into a sequence of subproblems that are amenable to convex optimization or closed-form solutions. We prove convergence of PALM to a critical point by leveraging the KL theory. 

We believe that our proof techniques and the upper bounds on the Lipschitz constants can be instrumental in developing other decomposition algorithms in large-scale power networks. While our model focuses on active AC power flows, the dynamics for the reactive power flows can be captured by the same set of nonlinear equations. We anticipate that the developed approach can be applied to fully nonlinear models with both active and reactive power equations. 

\section*{Acknowledgments}
We thank the reviewers for their comments and suggestions that improve the paper. This material is based upon work supported by the U.S. Department of Energy, Office of Science, Office of Advanced Scientific Computing Research, Applied Mathematics program under contract number DE-AC02-06CH11357.

\appendix

\subsection{Proof of Lemma~\ref{lem.gamsol}}
\label{pro.gamsol}

We prove by contradiction. Let $\gamma$ satisfy $\bfo^T \gamma = m-K$ and $\gamma_i \in \{0,1\}$, but $\gamma$ is different from the projection in~\eqref{eq.gamsol}. In other words, there exists at least one element of $\gamma$, say, the $l$th element such that $\gamma_l = 1$ with the corresponding $u^k_l < [u^k]_K$, and at least one element, say, the $j$th element such that $\gamma_j = 0$ with the corresponding $u_j^k \geq [u^k]_K$. Consider
\[
\delta_{lj} = (\gamma_l - u^k_l)^2 + (\gamma_j - u^k_j)^2 = (1 - u^k_l)^2 + (u^k_j)^2.
\]
and the cost of the swapping the values of $\gamma_l$ and $\gamma_j$
\[
\delta_{jl} = (u^k_l)^2 + (1 - u^k_j)^2.
\]
Since $\delta_{lj} - \delta_{jl} = 2(u^k_j - u^k_l) > 0$, we conclude that the cost function decreases if we choose $(\gamma_l = 0, \gamma_j = 1)$ instead of $(\gamma_l = 1, \gamma_j = 0)$. In other words, we can reduce the cost by swapping the values of $\gamma_l = 1$ with respect to $u^k_l < [u^k]_K$ and $\gamma_j = 0$ with respect to $u^k_l \geq [u^k]_K$ until~\eqref{eq.gamsol} is satisfied for all elements of $\gamma$. This completes the proof.

%\subsection{Derivation of~\eqref{eq.gradHthe}}
\subsection{Proof of Lemma~\ref{lem.gradH}}
\label{sec.gradH}
The derivations of~\eqref{eq.gradHgam} and~\eqref{eq.gradHz} are straightforward, as they amount to taking the derivatives of quadratic functions, thus omitted. The derivation of~\eqref{eq.gradHthe} involves taking the first-order variation for sine and cosine functions. We begin by taking variation $\tilde{\theta}$ around $\theta$ 
\begin{align*}
  \sin(E^T(\theta+\tilde{\theta}) )
  & \,=\, \sin(E^T\theta) \circ \cos(E^T \tilde{\theta}) \\
  & \,+\, \cos(E^T\theta) \circ \sin(E^T \tilde{\theta}) 
\end{align*}
where $\circ$ is the Hadamard (elementwise) product. When $\tilde{\theta}$ is small, we have the first-order approximation
\begin{align*}
  \sin(E^T(\theta+\tilde{\theta}))  
  & \,\approx \,
  \sin(E^T\theta)  \,+\, \diag(\cos(E^T \theta)) E^T \tilde{\theta}.
\end{align*}
It follows that the first-order approximation of $H_\rho(\theta+\tilde{\theta})$ is given by
\begin{align*}
  H_\rho(\theta+\tilde{\theta}) 
  & \approx H_\rho(\theta) + \rho \, ( ED\Gamma \sin(E^T\theta) - (P+z))^T \\
  & \hspace{0.5in} \times ED\Gamma
  \, \diag(\cos(E^T\theta)) E^T \tilde{\theta}.
\end{align*}
Taking the transpose of the matrix multiplying $\tilde{\theta}$ yields
\[
  \nabla_\theta H_\rho(\theta)
= \rho E \diag( \cos(E^T \theta) ) [Q \sin(E^T \theta) - R].
\]
where $Q = \Gamma D E^T E D \Gamma$ and $R = \Gamma D E^T(P + z)$. 

\subsection{Lipschitz constant of $\nabla_\theta H_\rho$}
\label{sec.lip}

Recall that
\[
  \| \sin(\theta_1 - \theta_2) \| \, \leq \, \| \theta_1 - \theta_2 \|
\]
for all $\theta_1,\theta_2$. We have 
\begin{align}
\nonumber
& \| \sin(E^T \theta_1) - \sin(E^T \theta_2) \| \\
\nonumber
\,&=\,  
\| 2 \cos(E^T (\theta_1+\theta_2)/2) \circ \sin(E^T (\theta_1-\theta_2)/2)\| \\
\label{eq.sin}
\,&\leq \,  2\|\sin(E^T (\theta_1 - \theta_2)/2)\| \leq \|E\| \|\theta_1 - \theta_2\|.
\end{align}
The equality is the elementwise sum-to-product identity. The first inequality follows from the fact that all cosine functions are upper bounded by 1. 
Similar calculation yields 
\begin{equation}
\label{eq.cos}
\| \cos(E^T \theta_1) - \cos(E^T \theta_2) \|  \leq \|E\| \|\theta_1 - \theta_2\|.
\end{equation}
Let $f(\theta) = \sin(E^T \theta)$ and $g(\theta) = \diag(\cos(E^T \theta)) Q$. By adding and subtracting the same term yields
\[
\begin{array}{l}
g(\theta_1) f(\theta_1) - 
  g(\theta_2) f(\theta_2)\\
=
g(\theta_1) (f(\theta_1) -  f(\theta_2)) +
(g(\theta_1) -  g(\theta_2)) f(\theta_2).
\end{array}
\]
We calculate 
\begin{align*}
  &
  \| g(\theta_1) f(\theta_1) -   g(\theta_2) f(\theta_2) \| \\
  & \leq \, \| \diag(\cos(E^T \theta_1)) Q (\sin(E^T \theta_1) - \sin(E^T \theta_2))\|  \\
  & \;+ \| \diag(\cos(E^T \theta_1) - \cos(E^T \theta_2)) Q \sin(E^T \theta_2) \| \\
  & \leq 2 \|Q\| \|E\| \| \theta_1 - \theta_2 \|
\end{align*}
where we have used~\eqref{eq.sin} and~\eqref{eq.cos}. It follows that the Lipschitz constant for $\nabla_\theta H_\rho$ is given by
\begin{align*}
  \| \nabla_\theta H_\rho(\theta_1) - \nabla_\theta H_\rho(\theta_2) \| 
\leq L_3 (\gamma^{k+1},z^{k+1}) \|\theta_1 - \theta_2\|,
\end{align*}
where 
\[
    L_3(\gamma^{k+1},z^{k+1}) =  \rho \|E\|^2 ( 2 \|Q^{k+1}\| + \|R^{k+1}\|)
\]
and 
\[
\begin{array}{l}
  Q^{k+1} \,=\, \Gamma^{k+1} D E^T E D \Gamma^{k+1},
  \\
  R^{k+1} \,=\, \Gamma^{k+1} DE^T (P+z^{k+1}).
\end{array}
\]

\bibliographystyle{IEEEtran}
\bibliography{../bibs/loadshed}
%%%%%%%%%%%%%%%%%%%%%%%%%%%%%%%%%%%%%%%%%%%%%%%%%%%%%%%%%%%%%%%%%%%%%%%%

\end{document}